\documentclass{amsart}
\usepackage{amssymb}
\usepackage[dvipdfmx]{graphicx}
\theoremstyle{plain}
\newtheorem{theorem}{Theorem}[section]
\newtheorem{lemma}[theorem]{Lemma}

\newtheorem{prop}[theorem]{Proposition}
\newtheorem{example}[theorem]{Example}
\theoremstyle{definition}
\newtheorem{definition}[theorem]{Definition}
\theoremstyle{remark}

\theoremstyle{equation}
\numberwithin{equation}{section}

\makeatletter
\def\@setcopyright{}
\def\serieslogo@{}
\makeatother
\thanks{This paper contains some results of my master's thesis submitted to Tokyo Institute of Technology.}
\date{\today}
\begin{document}
\author{Yukiko Abe}
\address{Department of Mathematics Tokyo Institute of Technology  Oh-okayama, Meguro, Tokyo 152-8551, Japan \\ e-mail: 09m00031@math.titech.ac.jp}
\title{The clock number of a knot}
\maketitle

\begin{abstract}
This paper is about the clock number of a knot. First we define the clock number by using states of a knot defined by Kauffman. 
Next we show that if $K$ is a prime knot, its clock number is greater than or equal to its crossing number. Finally we prove that its clock number is equal to its crossing number if and only if  $K$ is a two-bridge knot.
\end{abstract}

\section{Introduction}
The Alexander polynomial is an important and famous invariant of a knot. 
In \cite{MR712133}, Kauffman compose the Alexander polynomial by using the states which are obtained by the projection of a knot. 
The method of composing this polynomial is one of the most important results in \cite{MR712133}, but there are many definitions and theorems which are not considered in detail and may have useful informations. 
One example of them is about a lattice. A lattice is a graph which consists of states and is used for finding the states. 
We define the clock number of a knot by using a lattice graph and study it. Then we find a relation between the clock number and the crossing number. 
Moreover we prove that the clock number of a knot is equal to its crossing number if and only if it is a two-bridge knot, using basic knowledges of knot theory in \cite{MR0515288}.

\section{Preliminaries}

\begin{definition}[\cite{MR712133}]\label{def:universe_and_state}
\rm
A universe $U$ is an oriented four-valent planar graph which can be obtained as a projection of an oriented knot without over/under information. 

\begin{figure}[h]
 \begin{center}
 \includegraphics[scale=0.15]{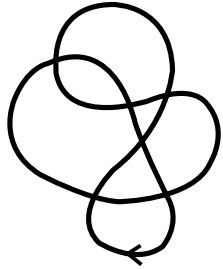}
 \caption{universe}
 \label{fig:universe}
 \end{center}
\end{figure}

A state of $U$ is an assignment of a marker to a corner of each vertex as in the form \raisebox{-2mm}{\includegraphics[scale=0.1]{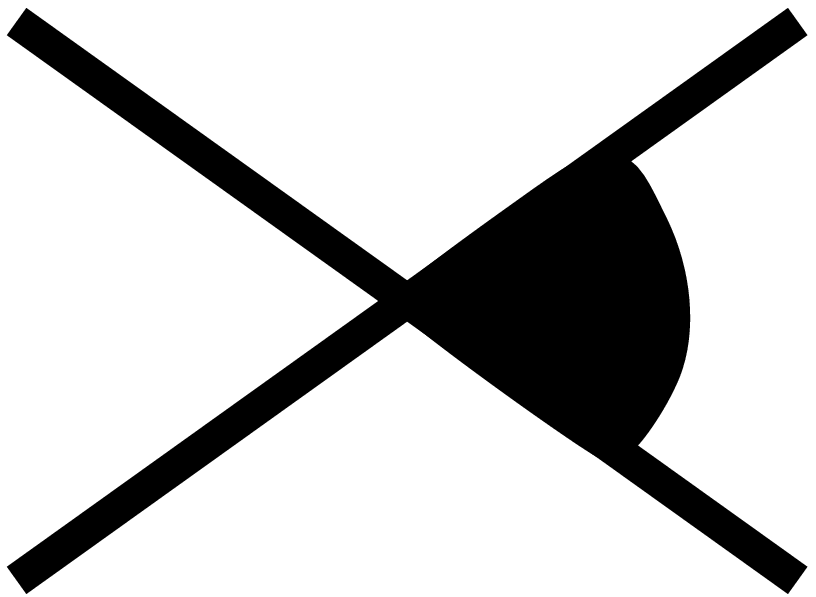}}, so that each region has exactly one marker, except for two adjacent regions where we put stars. See Figure \ref{fig:state} for example.
Note that the number of regions exceeds the number of vertices by two.

\begin{figure}[h]
 \begin{center}
 \includegraphics[scale=0.15]{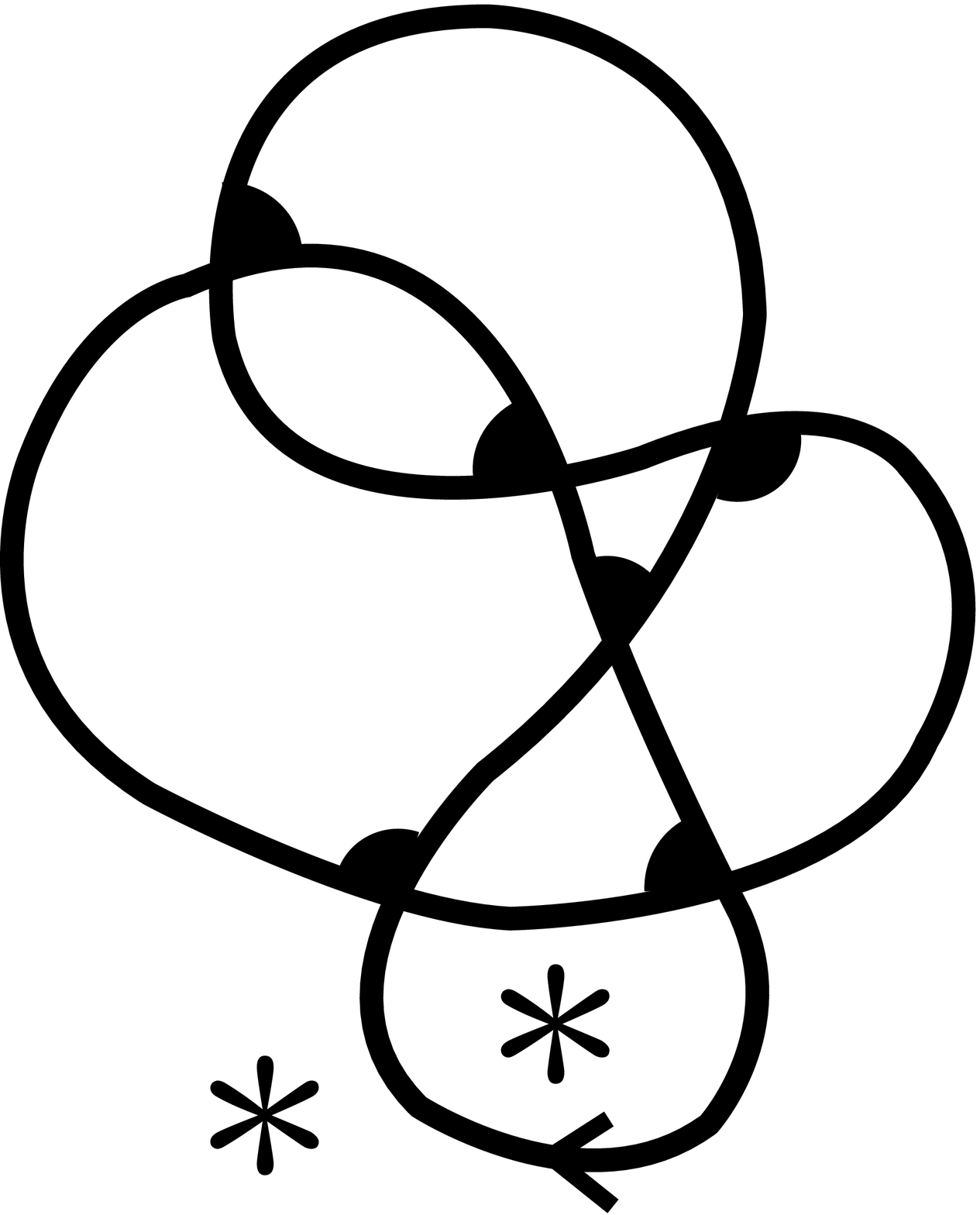}
 \caption{state}
 \label{fig:state}
 \end{center}
\end{figure}

\end{definition}

\begin{definition}
Let $K$ be a knot and $\tilde{K}$ be a diagram of $K$. 
If there exists a circle which has exactly two crossings with $\tilde{K}$, and we have at least one crossing both inside and outside of the circle, then we call the part enclosed in this circle a splittable part. 
When there is at least one splittable circle in the diagram, we call it a non-proper diagram. 
Similarly, when there is no splittable circle in the diagram, we call it a proper diagram. 

\begin{figure}[h]
 \begin{center}
 \includegraphics[scale=0.5]{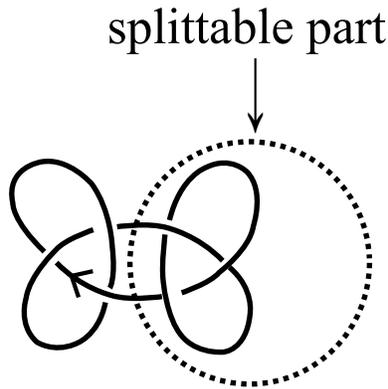}
 \caption{non proper diagram}
 \label{fig:non_prime_knot}
 \end{center}
\end{figure}

We also call a universe obtained by  a proper diagram a proper universe.  
\end{definition}


\begin{definition}
If a universe is proper, we call the first vertex the input point when we walk along the universe from the edge which the starred regions share, and the last vertex the output point. 
We also call a vertex touching one of the starred regions a boundary point.
Note that the input point and the output point are also boundary points (Figure \ref{fig:point_of_string}). 

\begin{figure}[h]
 \begin{center}
 \includegraphics[scale=0.2]{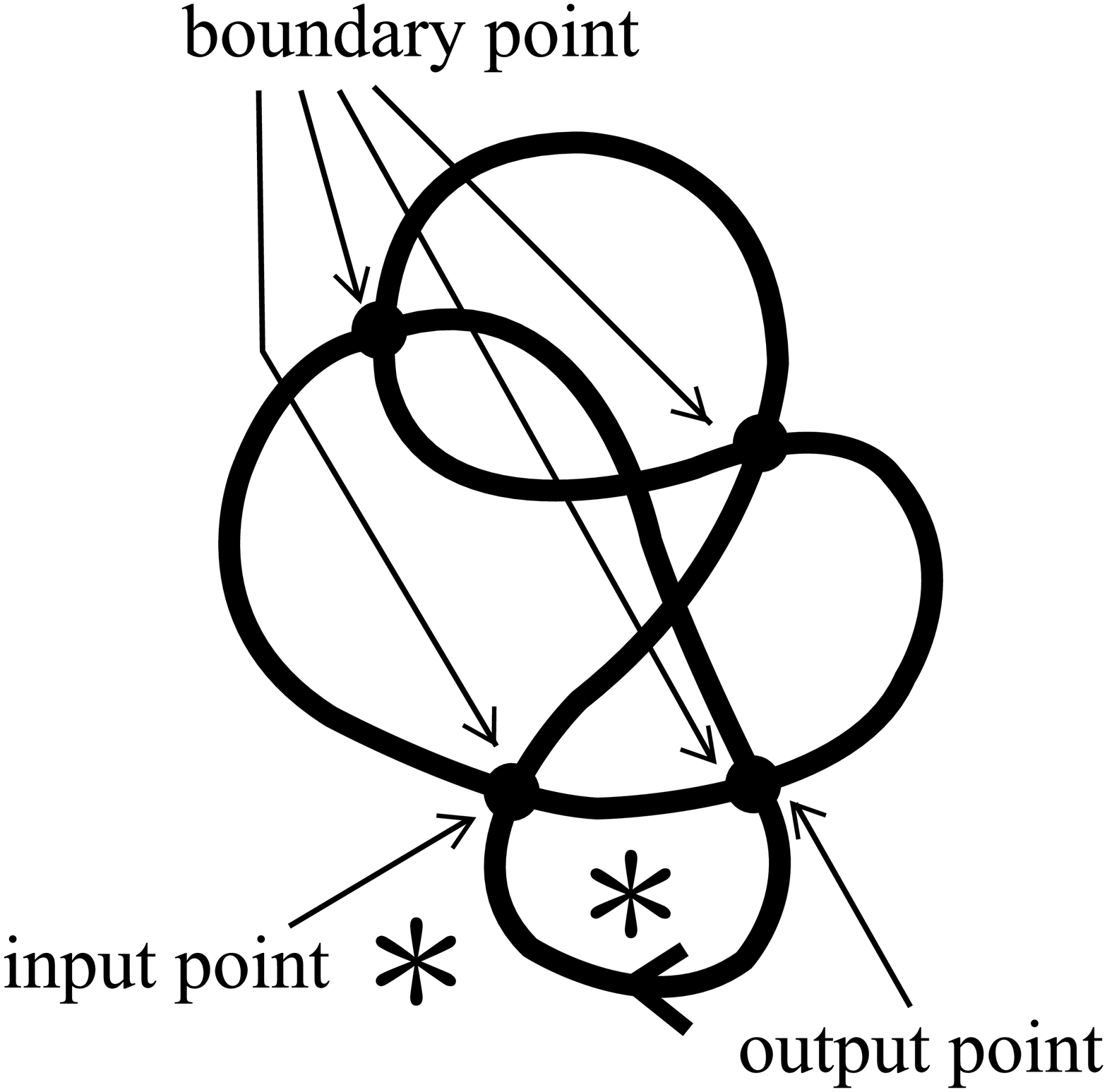}
 \caption{}
 \label{fig:point_of_string}
 \end{center}
\end{figure}
\end{definition}

\begin{definition}[\cite{MR712133}]\label{def:transposition}
\rm
Suppose that two regions share edges and the marker in each region is touching one of these edges.
Then a transposition is defined as a move switching these markers as in Figure \ref{fig:transposition}. 

\begin{figure}[h]
 \begin{center}
 \includegraphics[scale=0.3]{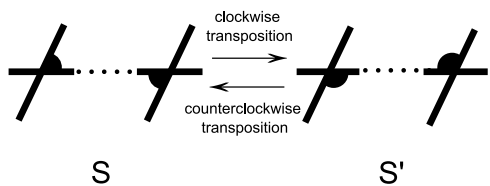}
 \caption{transposition}
 \label{fig:transposition}
 \end{center}
\end{figure}

If a state $S$ is obtained from $S'$ by moving the two markers clockwise, then we call it a clockwise transposition.
Similarly, the transposition from $S'$ to $S$ is a counterclockwise transposition. 

A clocked state is a state which admits only clockwise transpositions and a counterclocked state is a state which admits only counterclockwise transpositions.
See Figure \ref{fig:clocked_state}. 

\begin{figure}[h]
 \begin{center}
 \includegraphics[scale=0.35]{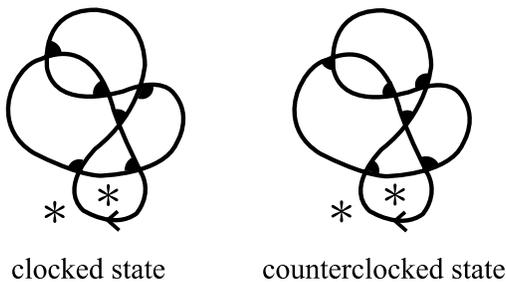}
 \caption{a clocked state and a counterclocked state}
 \label{fig:clocked_state}
 \end{center}
\end{figure}

\end{definition}

Using these definitions, Kauffman proved the following theorem.

\begin{theorem}[\cite{MR712133}]
\rm
Let $U$ be a universe and $\mathcal{S}$ be the set of all the states of $U$ with fixed stars in adjacent regions. Then we have the following.

\begin{enumerate}
\item The set $\mathcal{S}$ has a unique clocked state and a unique counterclocked state.
\item Any state in $\mathcal{S}$ can be reached from the clocked state by a series of clockwise transpositions.
\item Any state in $\mathcal{S}$ can be reached from the counterclocked state by a series of counterclockwise transpositions.
\end{enumerate}
\end{theorem}

We can draw an oriented graph with vertex set $\mathcal{S}$ and two vertices are connected by an arrow if they are related by a transposition, where the arrow indicates a clockwise transposition.

As in Figure \ref{fig:lattice}, we put the clocked state at the top, and the counterclocked state at the bottom. We call this graph the lattice of $U$ with a fixed starred regions. 

\begin{figure}[h]
 \begin{center}
 \includegraphics[scale=0.55]{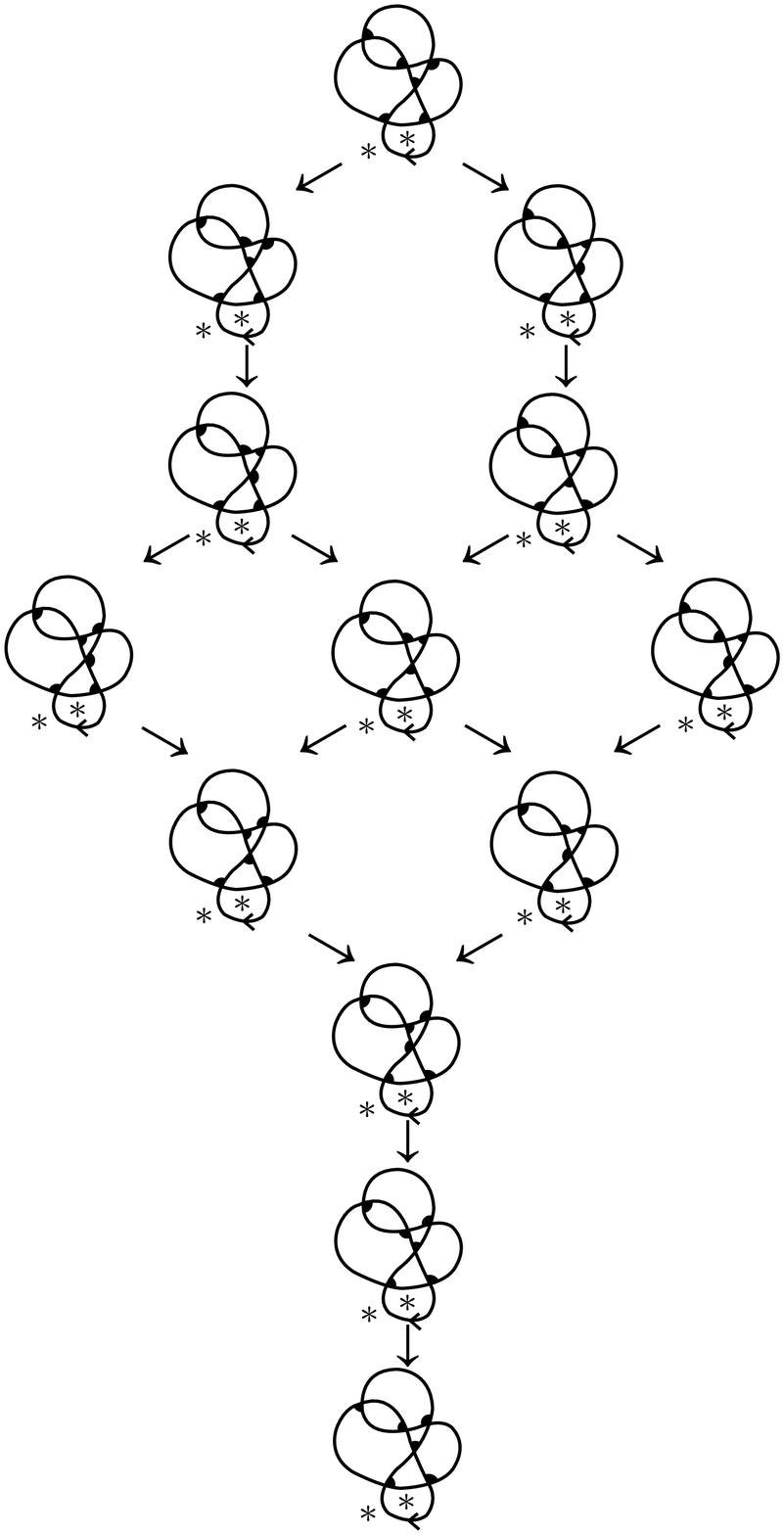}
 \caption{lattice}
 \label{fig:lattice}
 \end{center}
\end{figure}

\section{The definition of the clock number}
\begin{definition}\label{def:clock_number}
\rm
Let $K$ be a knot, $\tilde{K}$ be a diagram of $K$, $\hat{K}$ be the universe obtained from $\tilde{K}$, $R_{1}, R_{2}, \dots, R_{c(\tilde{K})+2}$ be the regions of $\hat{K}$, and $R_{i}$ and $R_{j}$ be adjacent regions, where $c(\tilde{K})$ is the number of crossings in $\tilde{K}$. 
Then we can generate the lattice of the states of $\hat{K}$ with stars in $R_{i}$ and $R_{j}$.
Let $p(\tilde{K}; i, j)$ denote the height of lattice, where the height means one plus the minimum number of transpositions which connect the clocked state and the counterclocked state. 
Define the clock number of $K$ by the fomula
\[
p(K) := \min\{ p(\tilde{K}; i, j)|  \hspace{0.1cm}\tilde{K}: \text{diagram,}\hspace{0.1cm} R_{i}, \hspace{0.1cm} R_{j}: \text{adjacent} \},
\]
where the number is taken over all diagrams of $K$ and adjacent regions.

\end{definition}

\section{A relation between the clock number and the crossing number}

\begin{theorem}\label{thm:relation_between_p_and_c}
\rm
Let $K$ be a prime knot. Then we have
\[
p(K) \geq c(K),
\]
where $c(K)$ is the minimum crossing number.
\end{theorem}

Before we prove Theorem \ref{thm:relation_between_p_and_c}, we introduce a lemma proved by Kauffman in \cite{MR712133}. 

\begin{lemma}[\cite{MR712133}]\label{lem:Kauffman_lemma}
\rm
Let $K$ be a knot, $\tilde{K}$ be a proper diagram, $\hat{K}$ be a universe obtained from $\tilde{K}$, $R_{1}, R_{2}, \dots, R_{c(\tilde{K})+2}$ be the regions of $\hat{K}$, and $R_{i}$ and $R_{j}$ be adjacent regions. 
Then the marker at the input point or the output point is transposed once, any marker at the boundary point (except for the input point and the output point) is transposed twice and any marker in the other cases is transposed four times or more when we get the counterclocked state from the clocked state.
\end{lemma}


We need another lemma. 

\begin{lemma}\label{lem:another_atom}
\rm
Let $K$ be a knot, $\tilde{K}$ be a non-proper diagram of $K$, and $\hat{K}$ be the universe obtained from $\tilde{K}$. 
Then we cannot transpose the marker at a vertex of one splittable part of $\hat{K}$ and the marker at a vertex of another splittable part of $\hat{K}$. 
\end{lemma}

\begin{proof}
We expand one of the starred regions of universe $\hat{K}$, and get Figure \ref{fig:lemma_atom_1}. 
\begin{figure}[h]
 \begin{center}
 \includegraphics[scale=0.25]{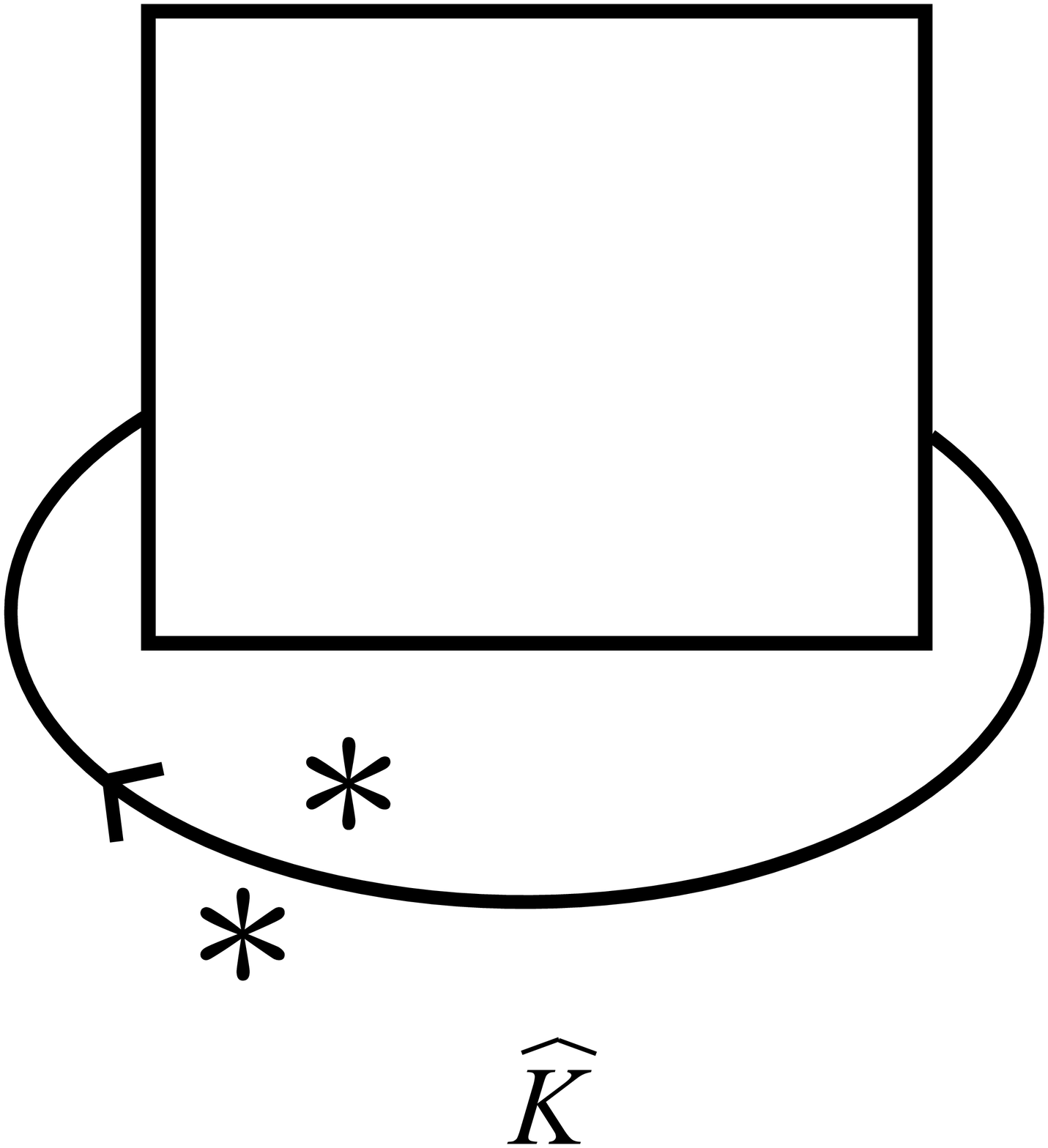}
 \caption{}
 \label{fig:lemma_atom_1}
 \end{center}
\end{figure}

Since $\hat{K}$ is a non-proper universe, we can draw Figure \ref{fig:lemma_atom_2}. 

\begin{figure}[h]
 \begin{center}
 \includegraphics[scale=0.25]{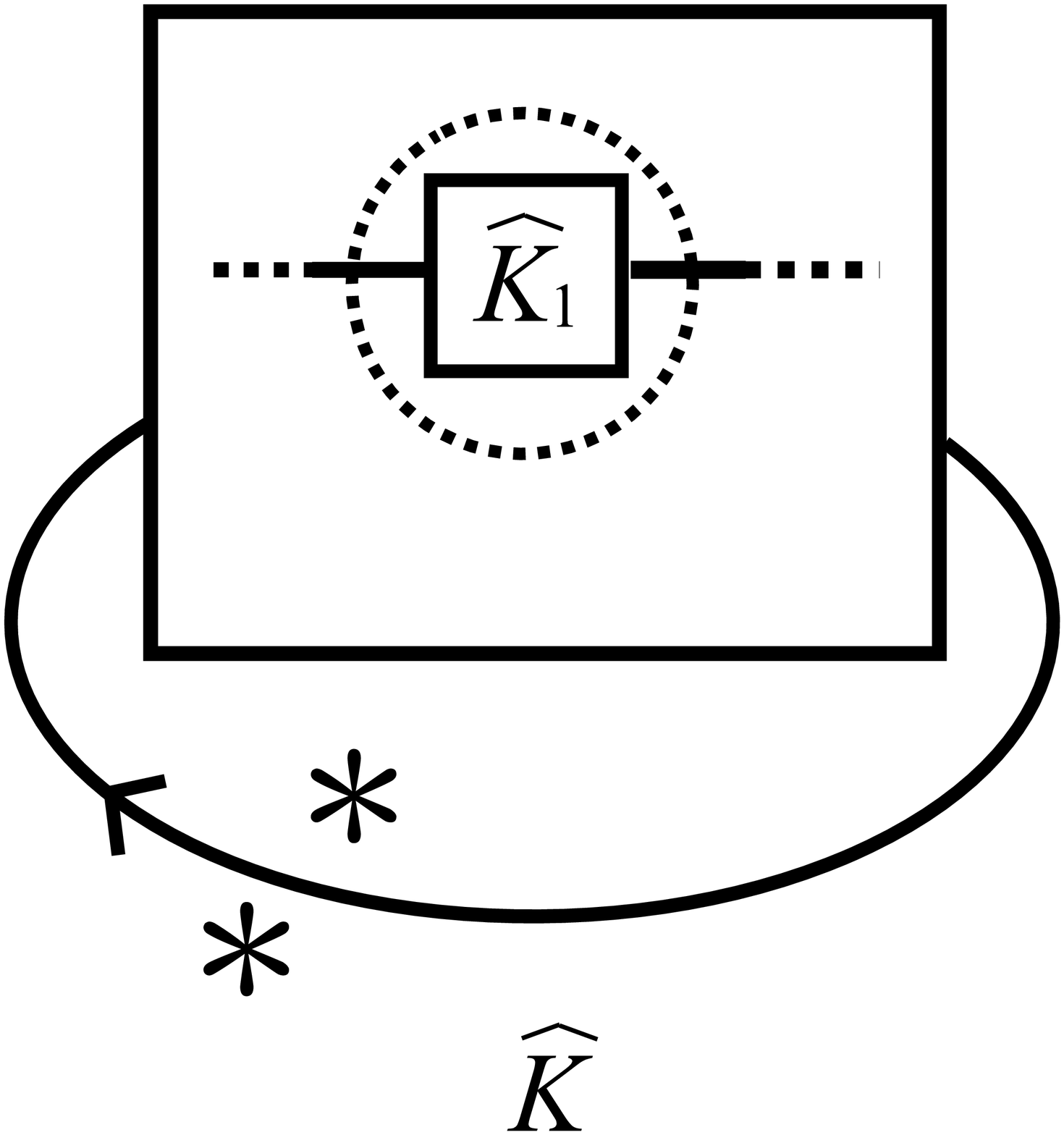}
 \caption{}
 \label{fig:lemma_atom_2}
 \end{center}
\end{figure}

Then the number of the regions inside the dotted circle which are not touching it is equal to the number of the vertices of $\hat{K_{1}}$. 
We call these regions $D_{1}, D_{2}, \dots, D_{n}$. 
Since every vertex and every region (except for two starred regions) has exactly one marker, we see that the markers at the vertices of $\hat{K_{1}}$ are in $D_{1}, D_{2}, \dots, D_{n}$. 

Next we think about the marker which can be transposed with the marker at a vertex of $\hat{K_{1}}$. 
The regions $D_{1}, D_{2}, \dots, D_{n}$ share edges only with the regions inside the dotted circle.  
Moreover every marker on an edge, which $D_{i}$ $(i=1, 2, \dots, n)$ and a region (with $(\ast)$) touching the dotted circle share, is in $D_{1}, D_{2}, \dots,$ or ${D_{n}}$. 
Therefore we can transpose the marker at a vertex of $\hat{K_{1}}$ only with a marker in $D_{1}, D_{2}, \dots$ or ${D_{n}}$.

This completes the proof. 
\end{proof}

\begin{proof}[Proof of \rm Theorem \ref{thm:relation_between_p_and_c}]
\rm
We choose $\tilde{K}$ so that $p(\tilde{K}; i, j) = p(K)$. 
If $\tilde{K}$ is not proper, then we have a splittable part (Figure \ref{fig:theorem1_1}). 
\begin{figure}[h]
 \begin{center}
 \includegraphics[scale=0.22]{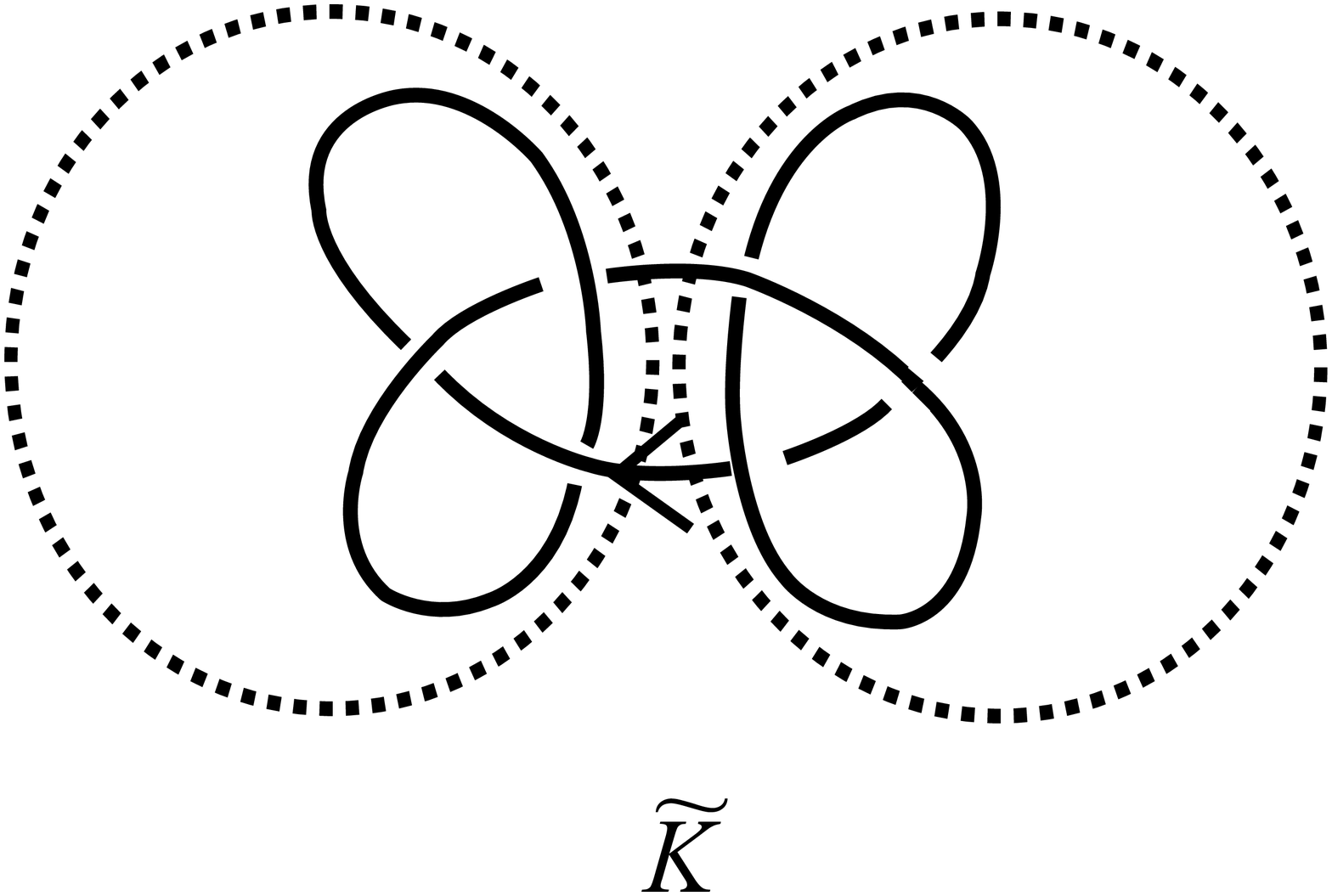}
 \caption{}
 \label{fig:theorem1_1}
 \end{center}
\end{figure}
Since $K$ is a prime knot, we may assume that there exists a splittable part in which we can remove all the crossings by using Reidemeister moves (Figure \ref{fig:theorem1_2}). 
Using Lemma \ref{lem:another_atom}, we have
\[
p(\tilde{K}; i, j) > p(\tilde{K}'; i, j).
\]
This fact contradicts the assumption $p(\tilde{K}; i, j) = p(K)$. 
\begin{figure}[h]
 \begin{center}
 \includegraphics[scale=0.22]{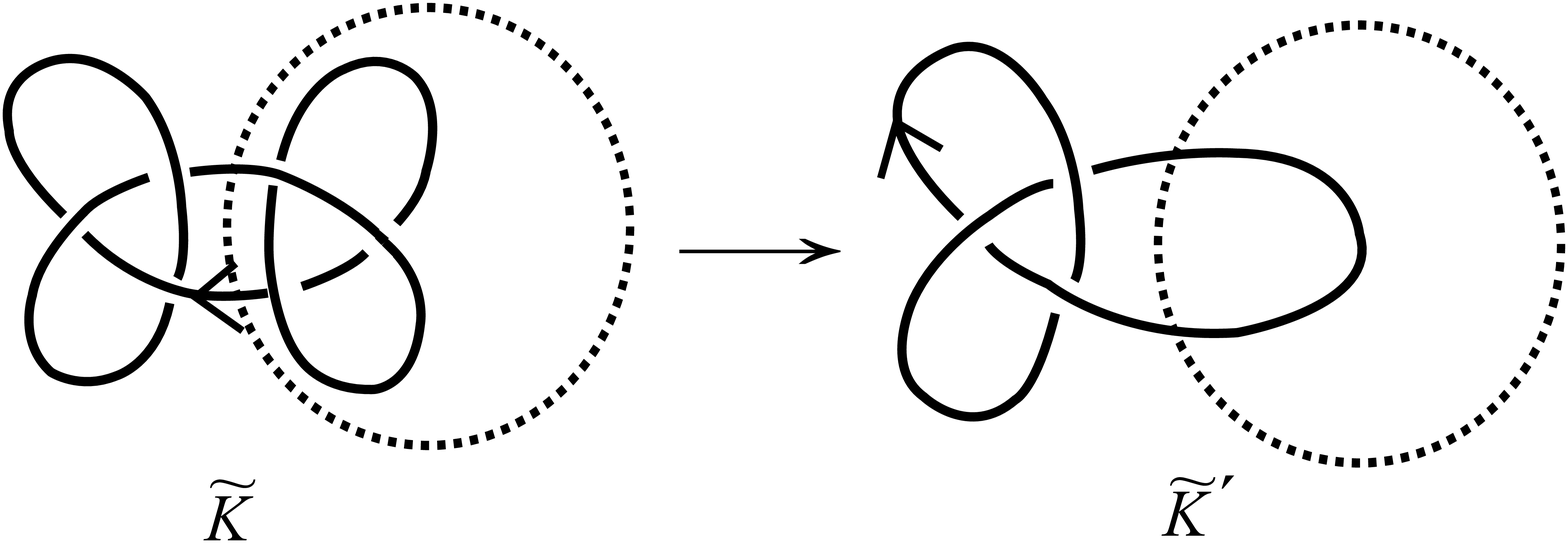}
 \caption{}
 \label{fig:theorem1_2}
 \end{center}
\end{figure}
Therefore $\tilde{K}$ is a proper diagram and $\hat{K}$ is a proper universe.

Since every proper universe has one input point and one output point, we see that the number of boundary points (except for the input point and the output point) is $c(\tilde{K})-2$ or less.
Moreover since $c(\tilde{K}) \geq c(K)$, then we have the following inequality.
\begin{equation}\label{eq:1}
\begin{split}
 p(K) &\geq 1 + \frac{1 \times 2 + 2(c(\tilde{K}) - 2)}{2} \\
      &= c(\tilde{K}) \\
      &\geq c(K),
\end{split}
\end{equation}
completing the proof.

\end{proof}


\section{The main theorem}

In this section we show that the equality of Theorem \ref{thm:relation_between_p_and_c} holds if and only if $K$ is a two-bridge knot. Actually we have the following theorem.

\begin{theorem}[Main Theorem]
\rm
Let $K$ be a prime knot. Then its clock number is equal to its crossing number if and only if $K$ is a two-bridge knot.
\end{theorem}

Before we prove the theorem, we will prove two lemmas and one proposition.

\begin{lemma}\label{lem:lemma1}
\rm
Let $K$ be a knot, $\tilde{K}$ be a proper diagram of $K$, $\hat{K}$ be the universe obtained from $\tilde{K}$, $R_{1}, R_{2}, \dots, R_{c(\tilde{K})+2}$ be regions of $\tilde{K}$, and $R_{i}$ and $R_{j}$ be adjacent regions.
Then the regions $R_{i}$ and $R_{j}$ share exactly two vertices.
\end{lemma}

\begin{proof}
\rm
We can draw $R_{i}$ and $R_{j}$ as in Figure \ref{fig:lemm1_figure_1}, because $R_{i}$ and $R_{j}$ share at least one edge and two vertices.
\begin{figure}[h]
 \begin{center}
 \includegraphics[scale=0.4]{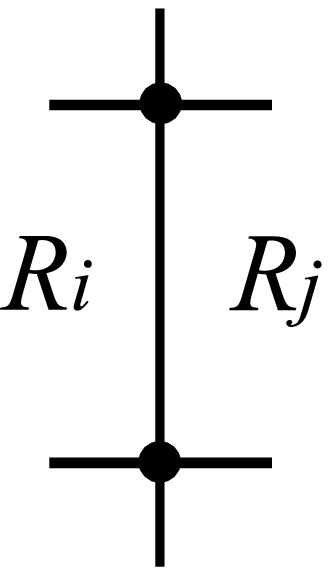}
 \caption{}
 \label{fig:lemm1_figure_1}
 \end{center}
\end{figure}
We assume for contradiction that they share one more vertex as in Figure \ref{fig:lemm1_figure_2}.
\begin{figure}[h]
 \begin{center}
 \includegraphics[scale=0.4]{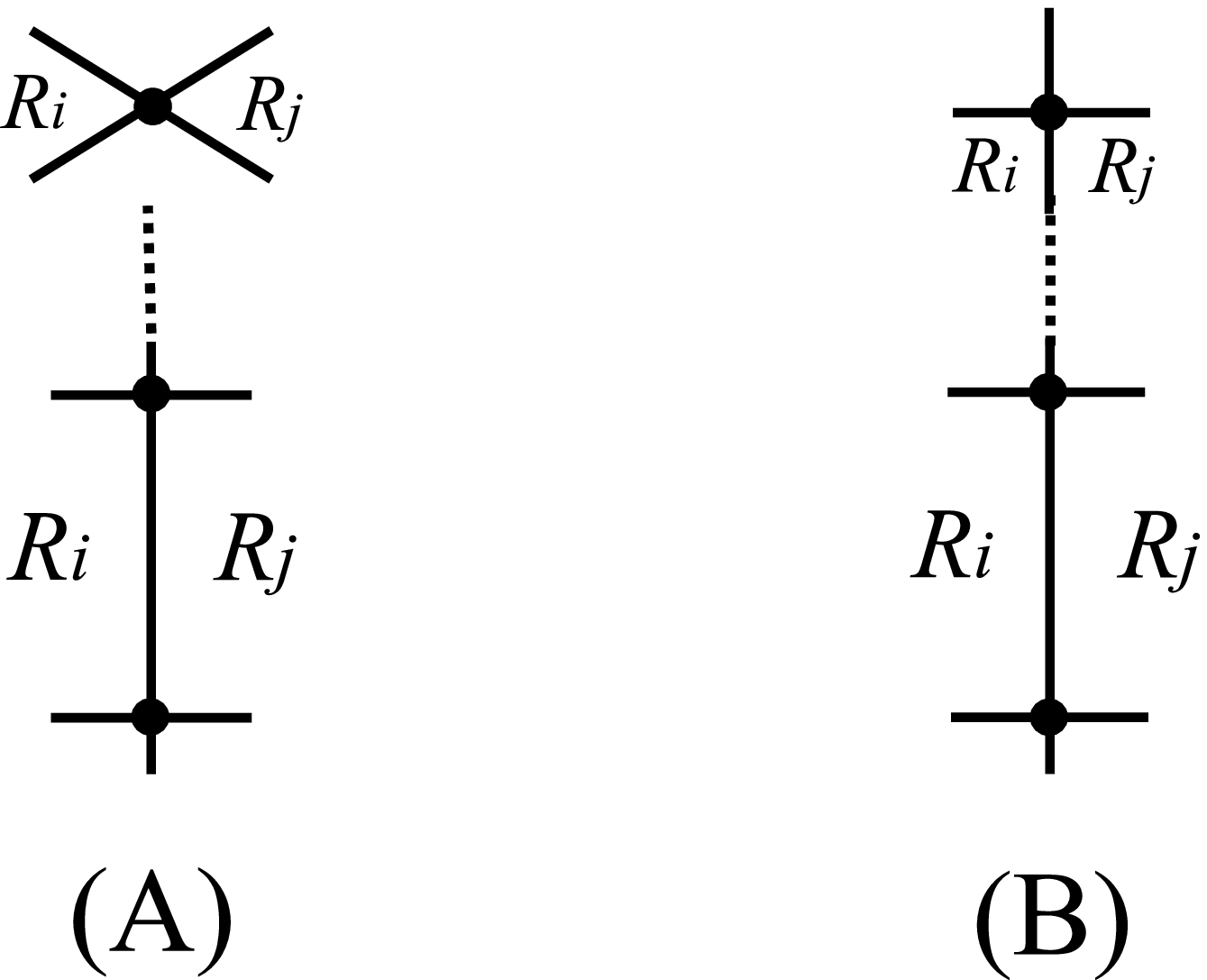}
 \caption{}
 \label{fig:lemm1_figure_2}
 \end{center}
\end{figure}
There are two cases: $({\rm A})$ $R_{i}$ and $R_{j}$ are in the opposite corners, $({\rm B})$ $R_{i}$ and $R_{j}$ are adjacent. In the case $({\rm A})$, this graph cannot be a universe because there is an odd number of edges in the dotted circle in Figure \ref{fig:lemm1_figure_3}.
In the case $({\rm B})$, this graph has a splittable part as Figure \ref{fig:lemm1_figure_3} contradicting the assumption that $\tilde{K}$ is proper. 

\begin{figure}[h]
 \begin{center}
 \includegraphics[scale=0.4]{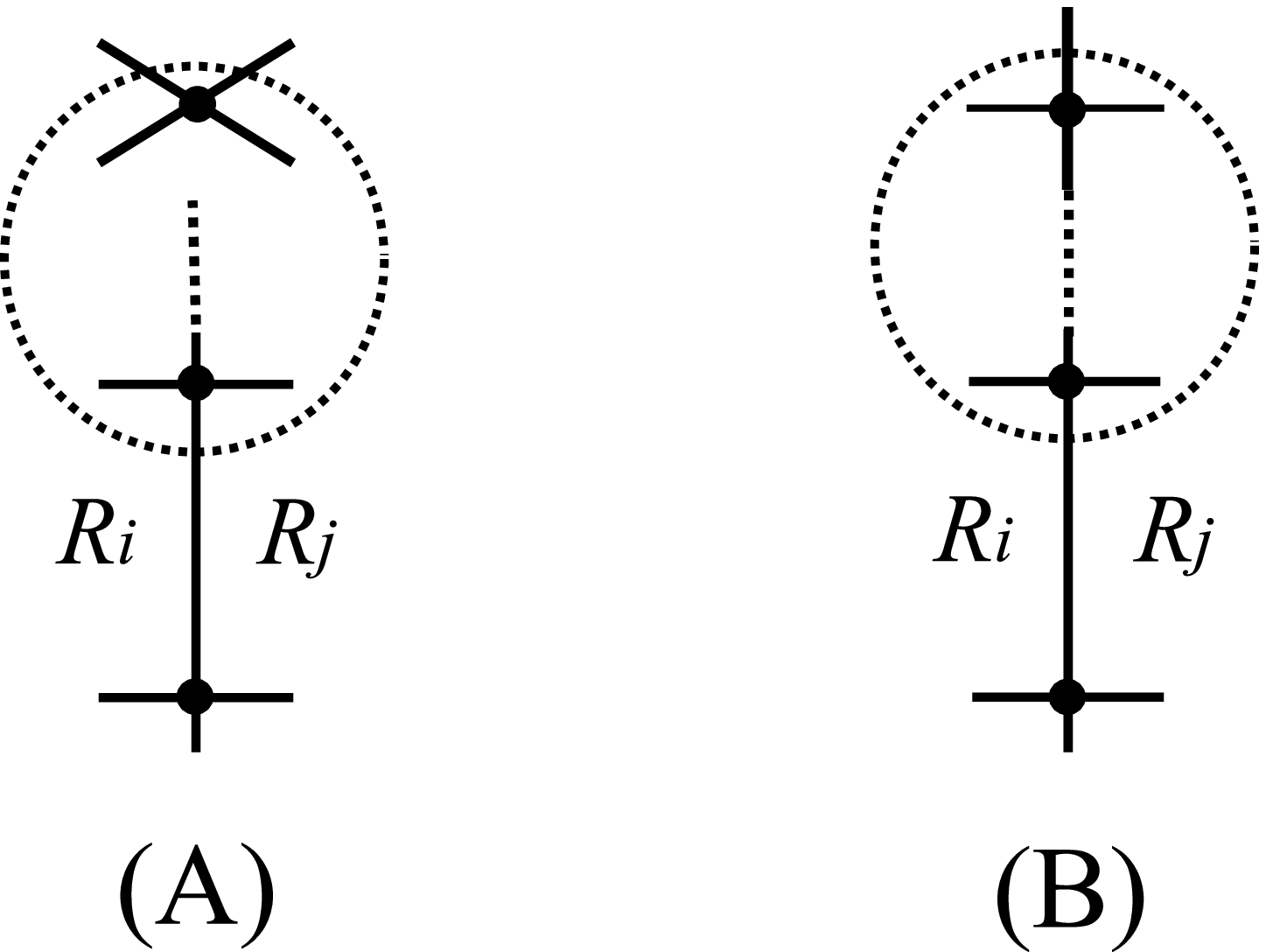}
 \caption{}
 \label{fig:lemm1_figure_3}
 \end{center}
\end{figure}

This is a contradiction, proving the lemma. 
\end{proof}


\begin{lemma}\label{lem:lemma2}
\rm
Let $K$ be a knot, $\tilde{K}$ be a proper diagram of $K$, $\hat{K}$ be the universe obtained from $\tilde{K}$, $R_{1}, R_{2}, \dots, R_{c(\hat{K})+2}$ be regions of $\tilde{K}$, $R_{i}$ and $R_{j}$ be adjacent regions, and $r_{i}$ ($r_{j}$, respectively) be the number of vertices of $R_{i}$ ($R_{j}$, respectively).
If $p(K)=p(\tilde{K}; i, j)$, then we have the following inequality.
\[
p(K)+r_{i}+r_{j} \geq 2c(K)+2.
\]

\end{lemma}

\begin{proof}
\rm
From Lemma \ref{lem:lemma1}, we see that $r_{i}+r_{j} - 2$ vertices around the starred regions are boundary points. 
Since every proper universe has one input point and one output point, there are $r_{i}+r_{j} - 4$ boundary points exept for the input point and the output point. 
From Lemma \ref{lem:Kauffman_lemma}, two markers are transposed once, $r_{i}+r_{j} - 4$ markers are transposed twice and the other markers are transposed four times or more. 
Thus we have
\begin{equation*}
\begin{split}
 p(K) &\geq \frac{1 \times 2 + 2(r_{i}+r_{j} - 4) + 4(c(\tilde{K}) - (r_{i}+r_{j} - 2))}{2} +1 \\
      &= - r_{i} - r_{j} + 2c(\tilde{K}) + 2.
\end{split}
\end{equation*}
Using $c(\tilde{K}) \geq c(K)$, we obtain. 
\[
p(K)+r_{i}+r_{j} \geq 2c(K)+2.
\]
This proves the lemma.
\end{proof}


\begin{prop}\label{prop:prop1}

\rm
Let $K$ be a knot, $\tilde{K}$ be a proper diagram of $K$, $\hat{K}$ be the universe obtained from $K$,  $R_{1}, R_{2}, \dots, R_{c(\tilde{K})+2}$ be regions of $\tilde{K}$, and $r_{k}$ be the number of vertices of $R_{k}$ ($k=1, 2, \dots, c(\tilde{K})+2$). 
If there exist adjacent regions $R_{i}$ and $R_{j}$ with $r_{i} + r_{j} =c(\tilde{K}) + 2$, then $K$ is a two-bridge knot.

\end{prop}

\begin{proof}
\rm
Since two adjacent regions $R_{i}$ and $R_{j}$ share exactly two vertices (from Lemma \ref{lem:lemma1}) and $r_{i} + r_{j} = c(\tilde{K}) + 2$, a part of $\hat{K}$ looks like Figure \ref{fig:prop_figure_1}.
Here we draw all the vertices and some edges.

\begin{figure}[h]
 \begin{center}
 \includegraphics[scale=0.4]{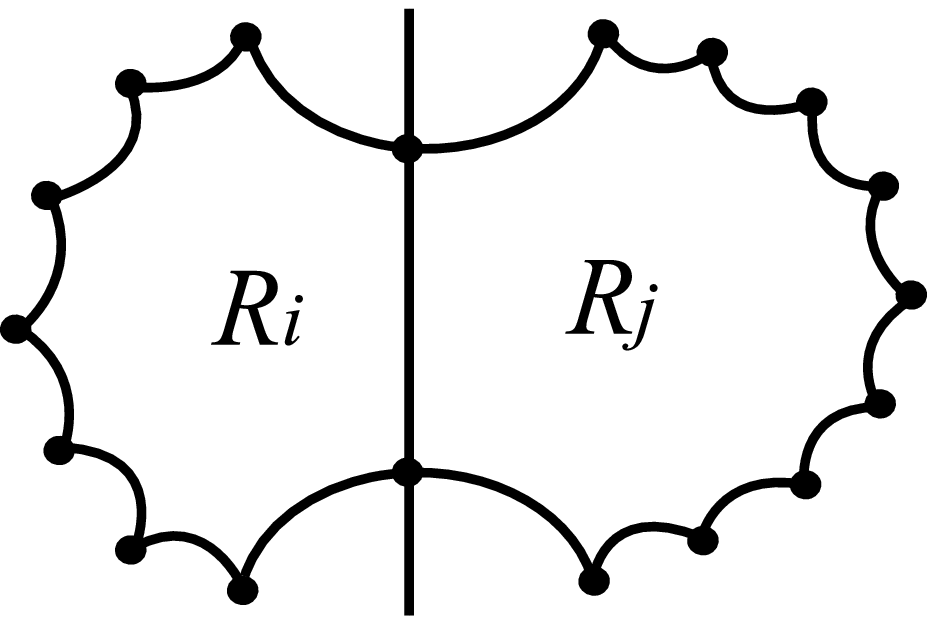}
 \caption{}
 \label{fig:prop_figure_1}
 \end{center}
\end{figure}

To restore $\hat{K}$ we need to add edges following the three conditions below.

\begin{enumerate}
\item There are four edges per one vertex.
\item The resulting universe has no splittable part.
\item Any new edge makes no more vertices.
\end{enumerate}

We will show that the resulting universe is a two-bridge knot projection.

We start drawing an edge from one of the vertices which $R_{i}$ and $R_{j}$ share (See Figure \ref{fig:prop_figure_2}).

\begin{figure}[h]
 \begin{center}
 \includegraphics[scale=0.4]{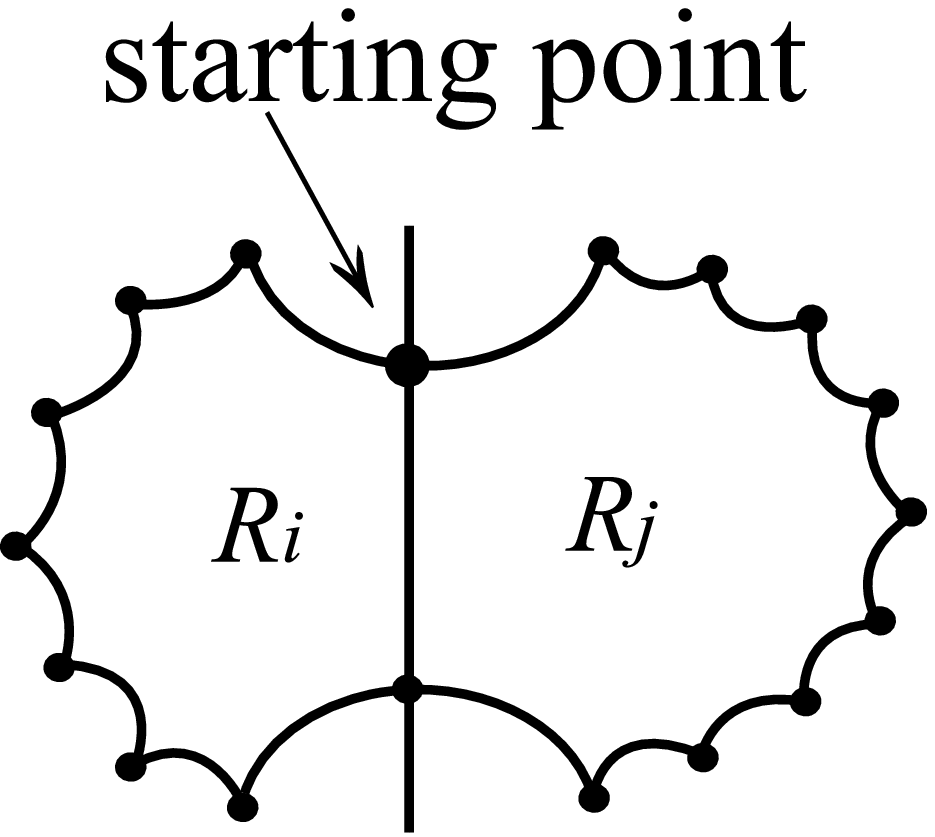}
 \caption{}
 \label{fig:prop_figure_2}
 \end{center}
\end{figure}

Then the end point of this edge should be either the right-most point in the left hand side or the left-most point in the right hand side, since otherwise we would have a splittable part as in Figure \ref{fig:prop_figure_4}. 

\begin{figure}[h]
 \begin{center}
 \includegraphics[scale=0.4]{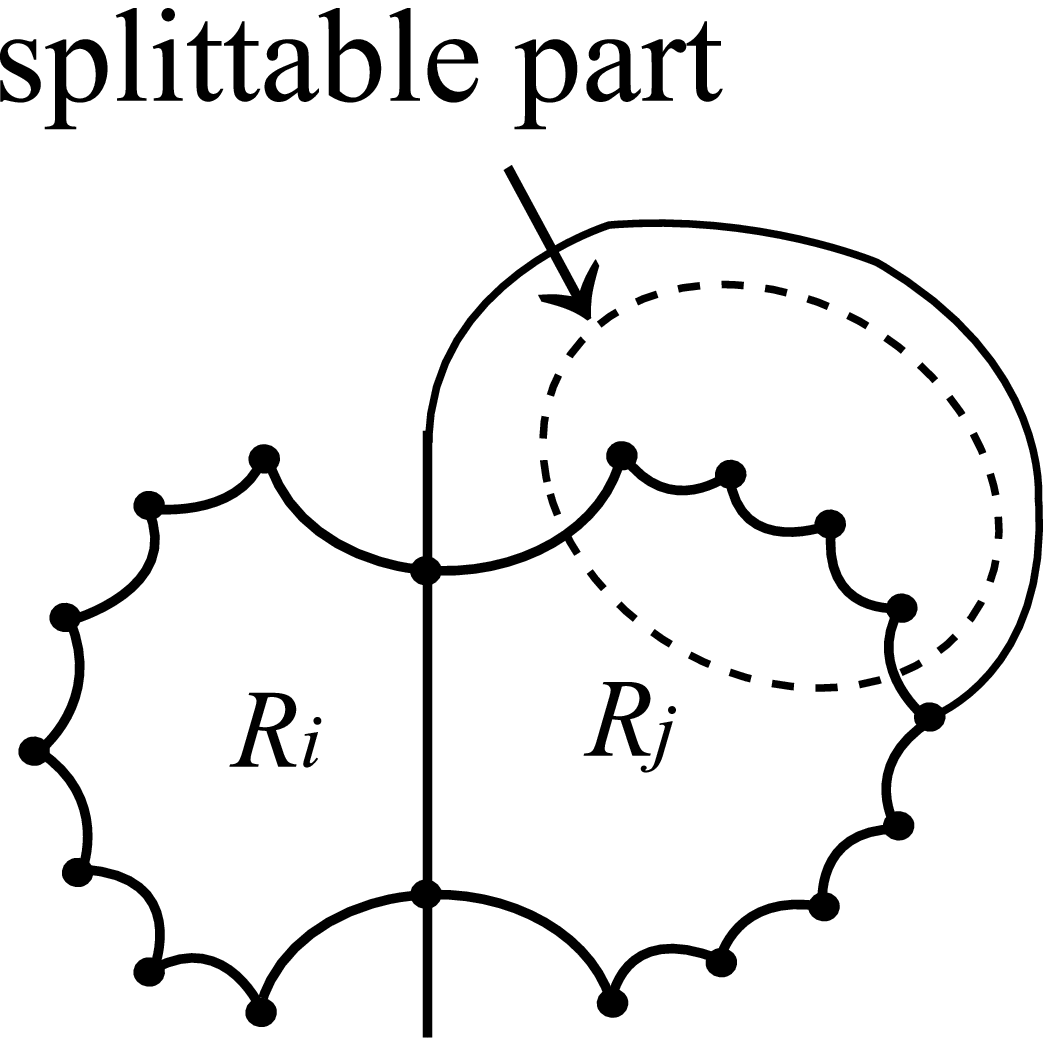}
 \caption{}
 \label{fig:prop_figure_4}
 \end{center}
\end{figure}

In Figure \ref{fig:prop_figure_5}, we choose the point in the right side of the starting point for example.

\begin{figure}[h]
 \begin{center}
 \includegraphics[scale=0.4]{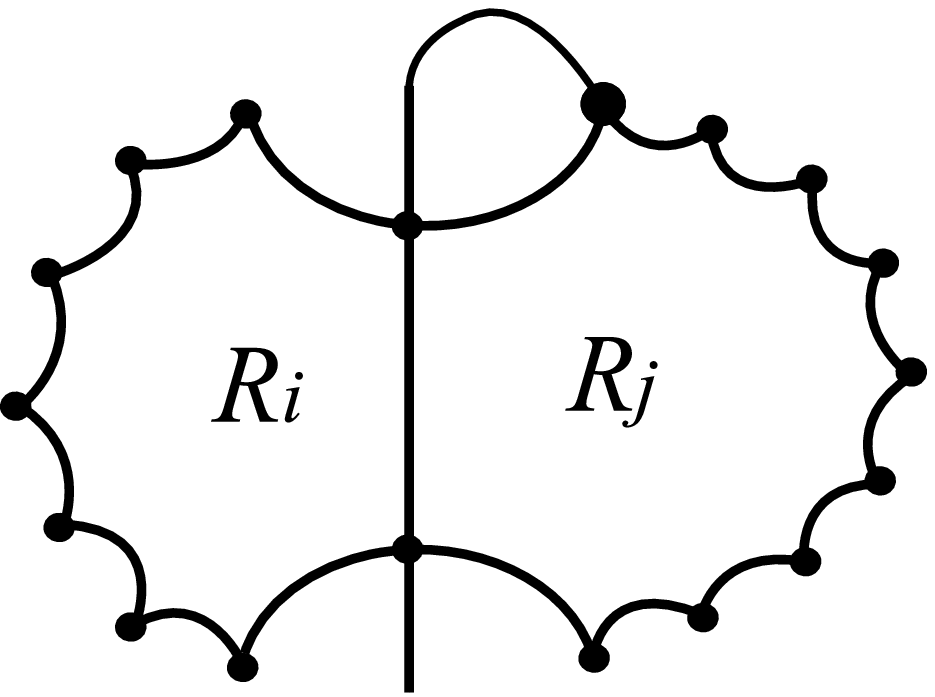}
 \caption{}
 \label{fig:prop_figure_5}
 \end{center}
\end{figure}

Next we draw another  edge from the end point of the first edge.
Then the end point of this edge should be the vertex next to the starting point of this edge or the vertex next to the starting point of the first edge (Figure \ref{fig:prop_figure_6}). 
\begin{figure}[h]
 \begin{center}
 \includegraphics[scale=0.4]{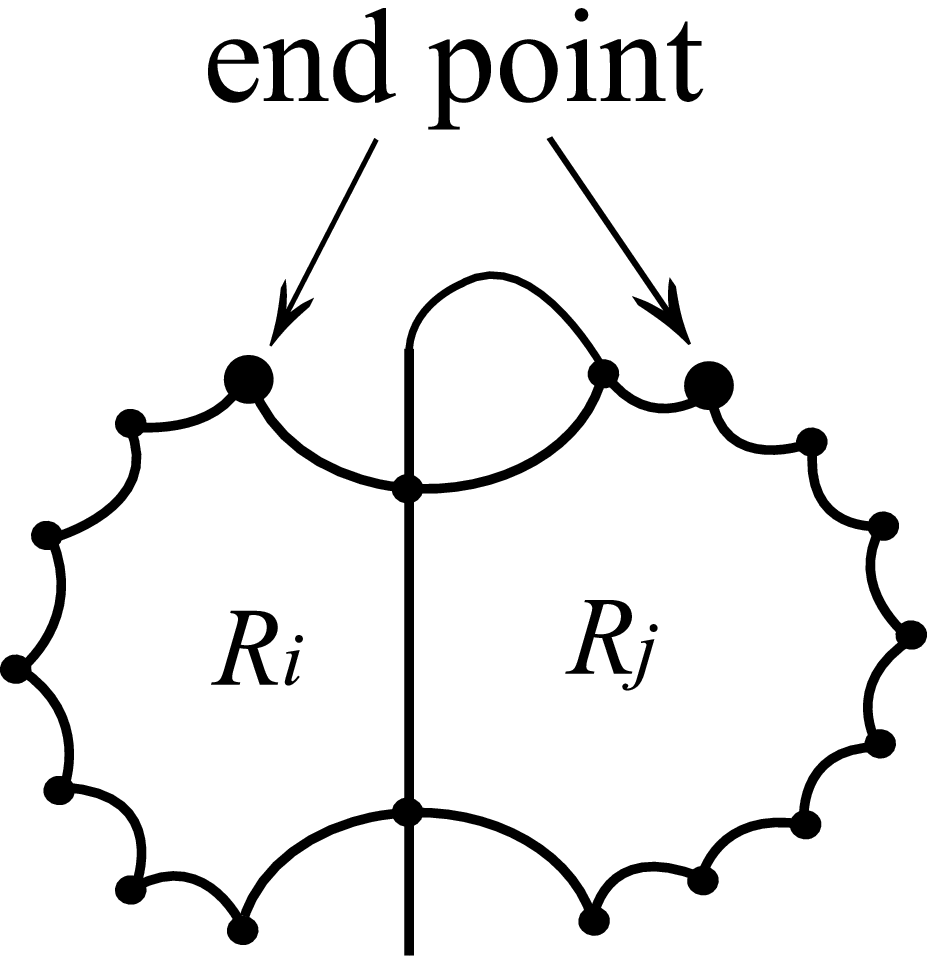}
 \caption{}
 \label{fig:prop_figure_6}
 \end{center}
\end{figure}
In Figure \ref{fig:prop_figure_7}, we choose the point in the right side of the starting point of this edge for example.

\begin{figure}[h]
 \begin{center}
 \includegraphics[scale=0.4]{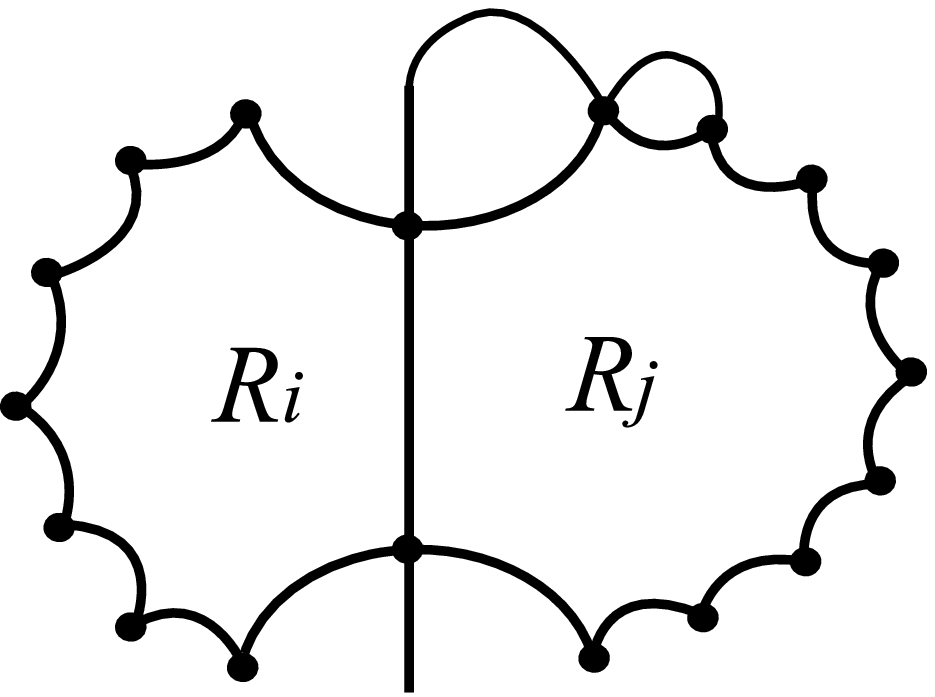}
 \caption{}
 \label{fig:prop_figure_7}
 \end{center}
\end{figure}

When we continue drawing edges, each edge has two choices for its end point.
Therefore we obtain Figure \ref{fig:prop_figure_8} as an example.

\begin{figure}[h]
 \begin{center}
 \includegraphics[scale=0.4]{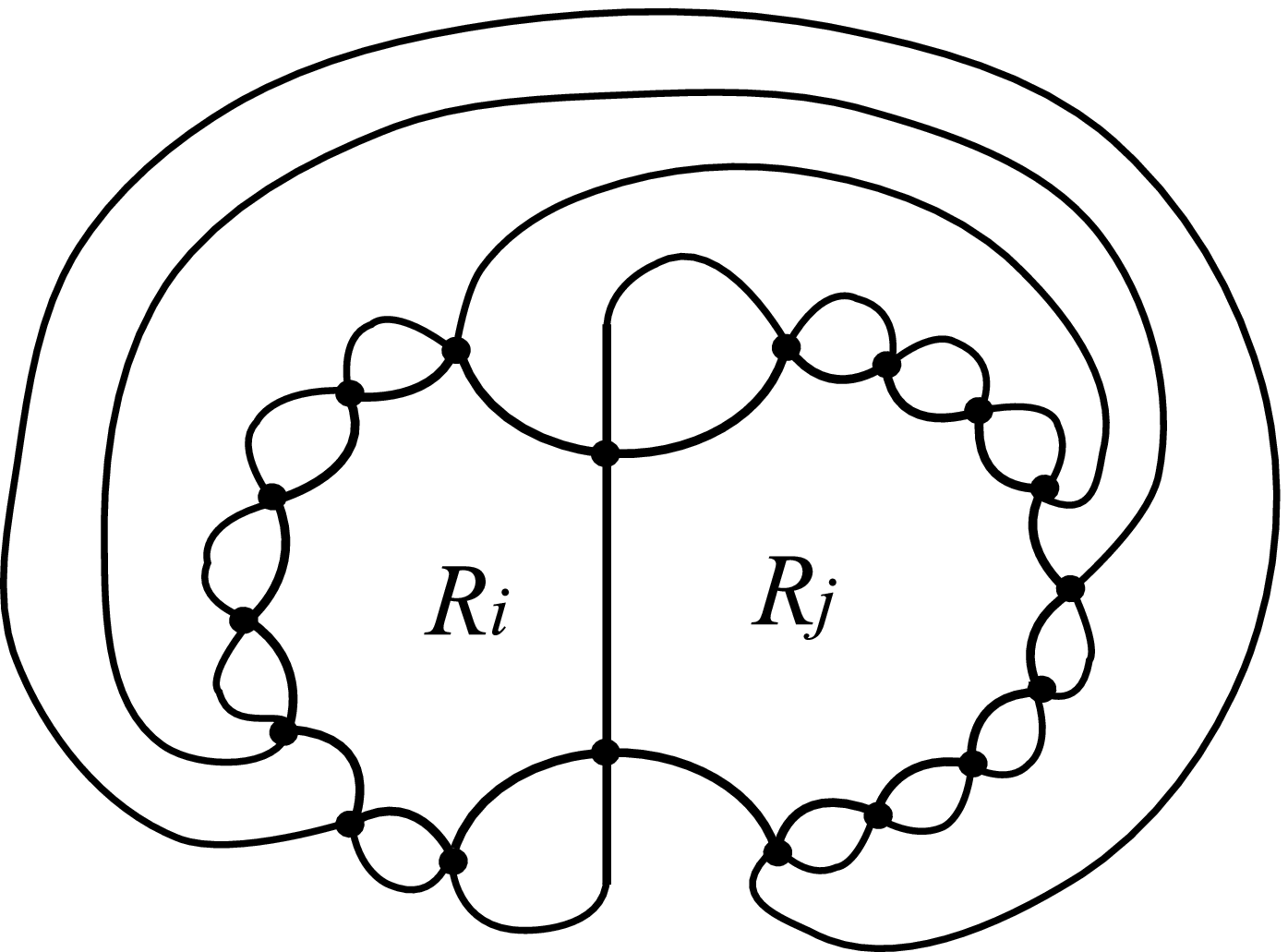}
 \caption{}
 \label{fig:prop_figure_8}
 \end{center}
\end{figure}

As in Figure \ref{fig:prop_figure_9}, we regard a twisted part \raisebox{-1mm}{\includegraphics[scale=0.1]{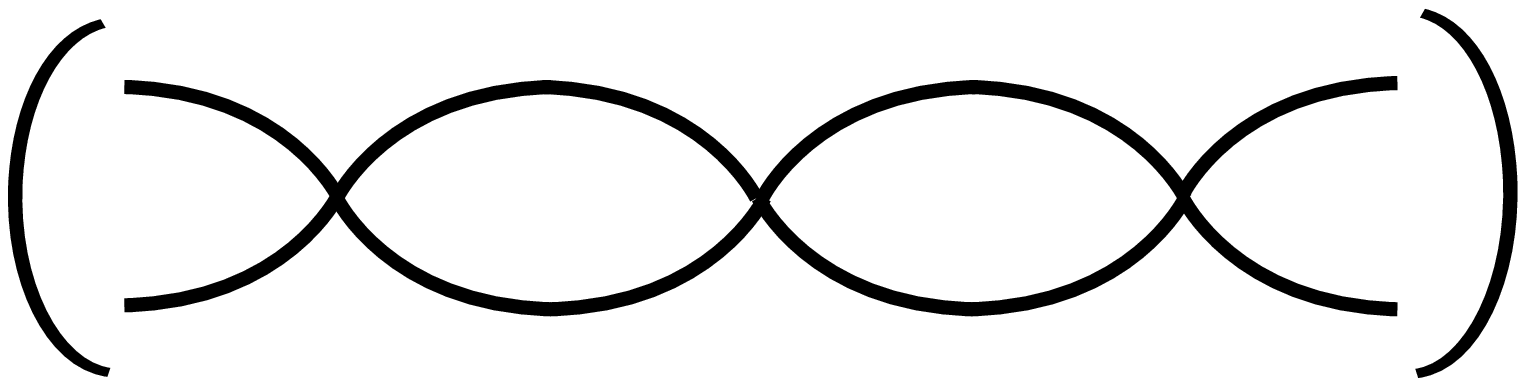}} as a box.

\begin{figure}[h]
 \begin{center}
 \includegraphics[scale=0.4]{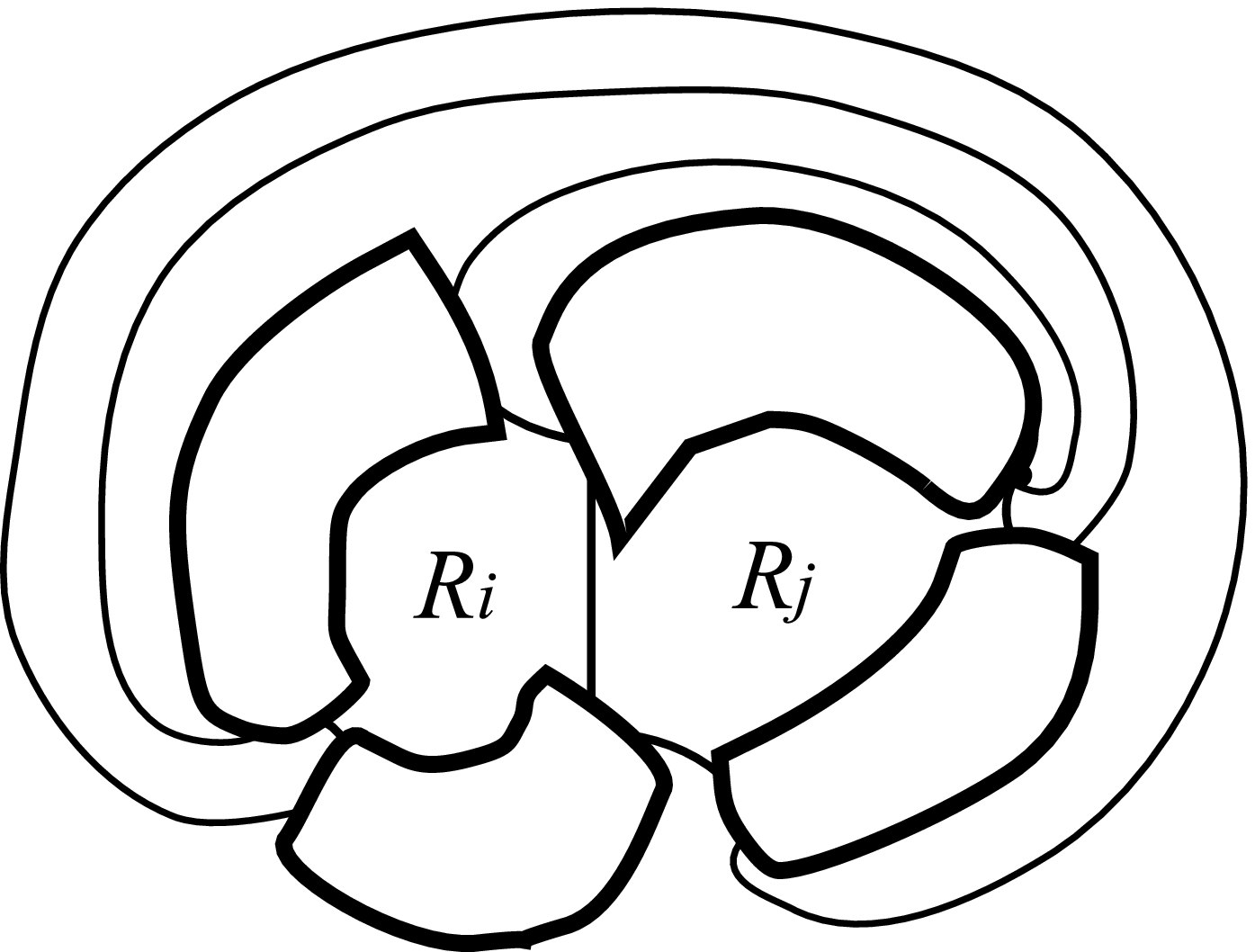}
 \caption{}
 \label{fig:prop_figure_9}
 \end{center}
\end{figure}

Expanding $R_{i}$ on the sphere and smoothing the edges, we obtain a universe $\hat{K}$ as in Figure \ref{fig:prop_figure_10}.

\begin{figure}[h]
 \begin{center}
 \includegraphics[scale=1.0]{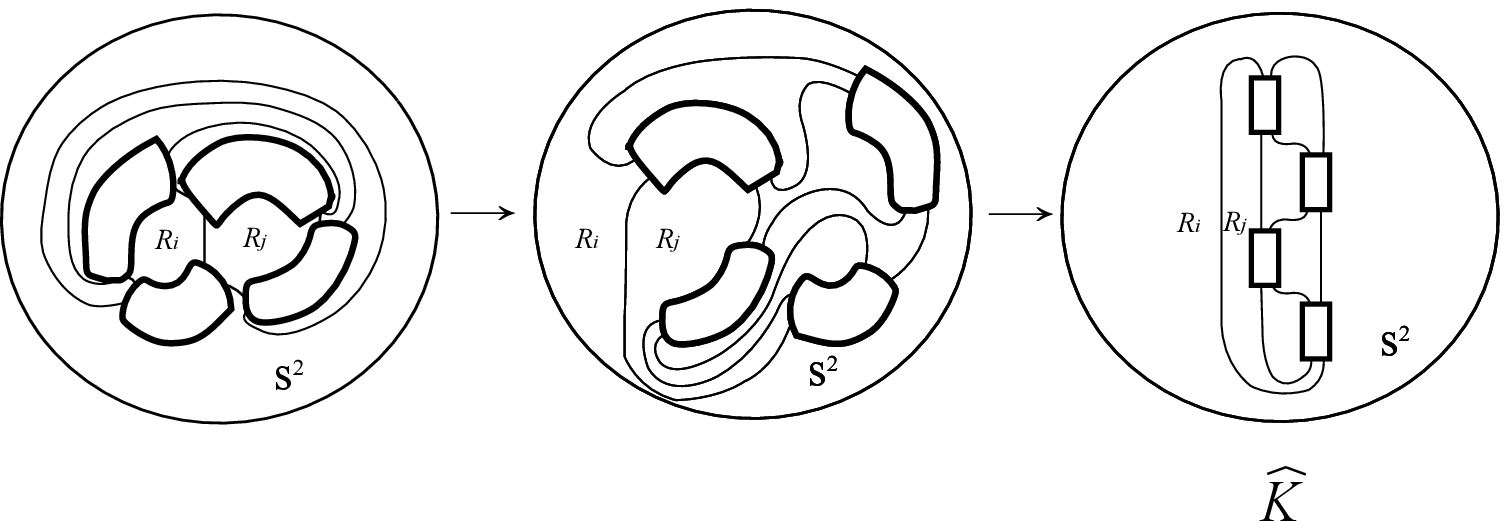}
 \caption{}
 \label{fig:prop_figure_10}
 \end{center}
\end{figure}

We see that $\hat{K}$ is a two-bridge knot projection.
Then $K$ is a two-bridge knot or the unknot.

This completes the proof.

\end{proof}

We now return to the proof of the main theorem.

\begin{proof}[Proof of \rm Main Theorem]
\rm
We suppose that $p(K)=c(K)$. We choose $\tilde{K}$, $R_{i}$ and $R_{j}$ so that $p(\tilde{K}; i, j) = p(K)$, where $R_{i}$ and $R_{j}$ are starred regions.
By Lemma \ref{lem:lemma2}, we see that
\[
p(K) + r_{i}+r_{j} \geq 2c(K)+2,
\]
and hence we have
\[
r_{i} + r_{j} \geq c(K)+2.
\]
Moreover using Lemma \ref{lem:lemma1} and inequality \eqref{eq:1}, we see
\begin{equation*}
\begin{split}
 r_{i} + r_{j} - 2 &\leq c(\tilde{K}) \\
                   &\leq p(K) \\
                   &= c(K).
\end{split}
\end{equation*}
Then we have the inequality
\[
r_{i} + r_{j} \leq c(K) + 2.
\]
Therefore we obtain
\[
r_{i} + r_{j} = c(K) + 2.
\]
Using Proposition \ref{prop:prop1}, we see that $K$ is a two-bridge knot.

Conversely, if $K$ is a two-bridge knot, a diagram $\tilde{K}$ which has $c(K)=c(\tilde{K})$ is one of the following two diagrams in Figure \ref{fig:main_theorem_figure_1}.

\begin{figure}[h]
 \begin{center}
 \includegraphics[scale=0.4]{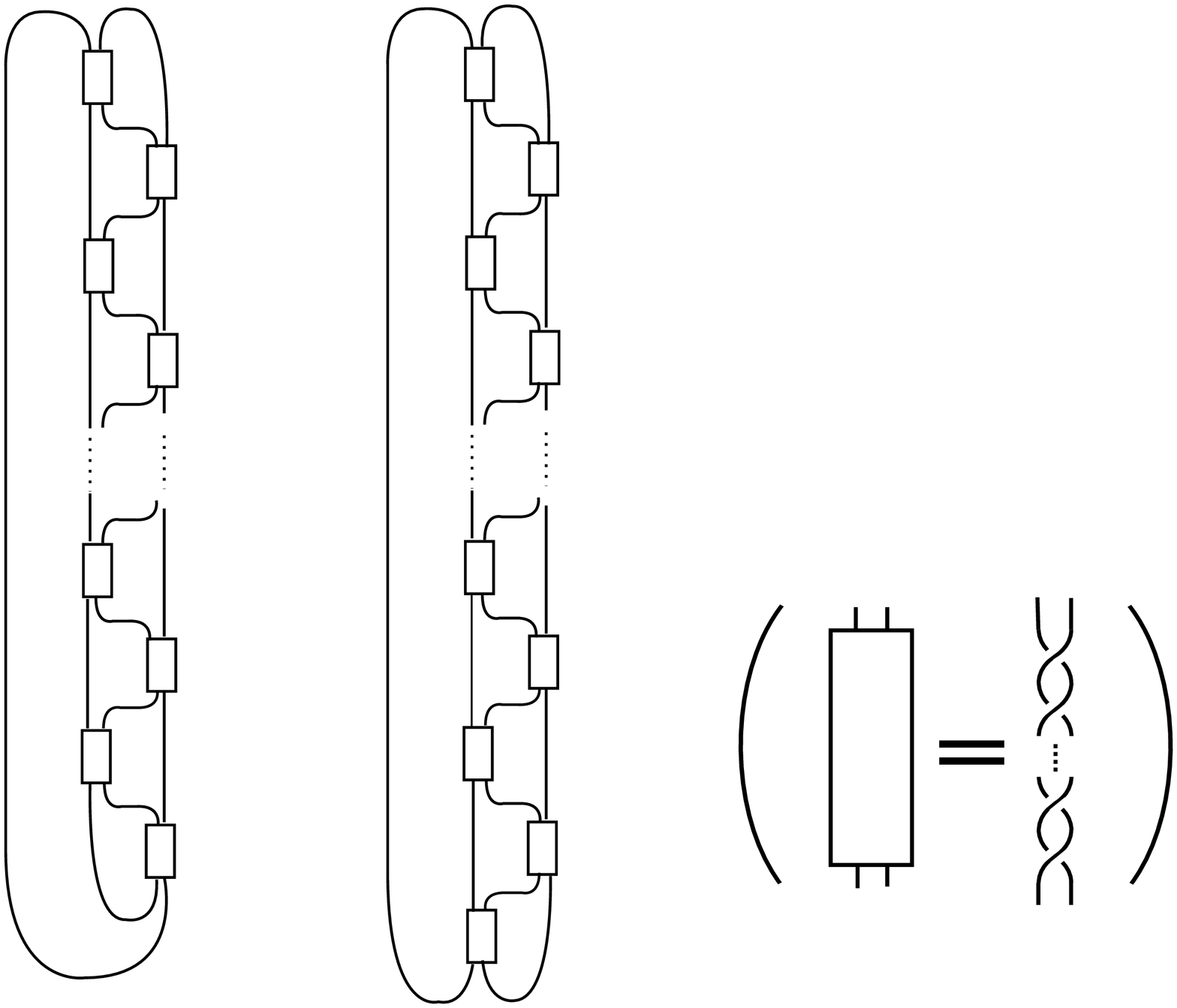}
 \caption{}
 \label{fig:main_theorem_figure_1}
 \end{center}
\end{figure}

We will show that $p(K)$ is equal to $c(K)$ when $\tilde{K}$ has an even number of boxes.
When we choose the starred regions as in Figure \ref{fig:main_theorem_figure_2}, the clocked state and the counterclocked state are as in Figure \ref{fig:main_theorem_figure_3}.

\begin{figure}[h]
 \begin{center}
 \includegraphics[scale=0.3]{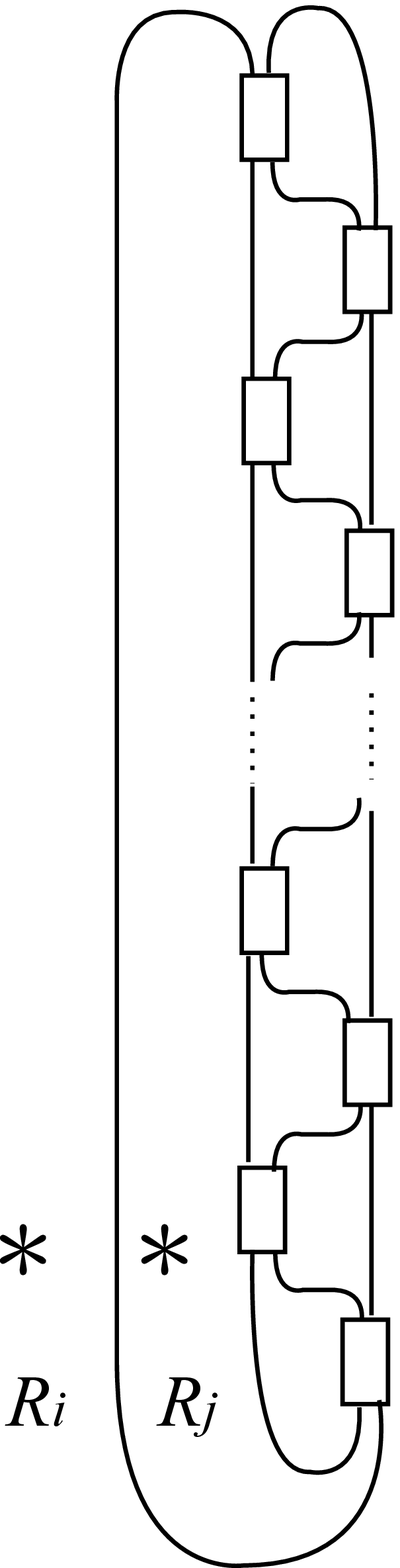}
 \caption{}
 \label{fig:main_theorem_figure_2}
 \end{center}
\end{figure}

\begin{figure}[h]
 \begin{center}
 \includegraphics[scale=0.8]{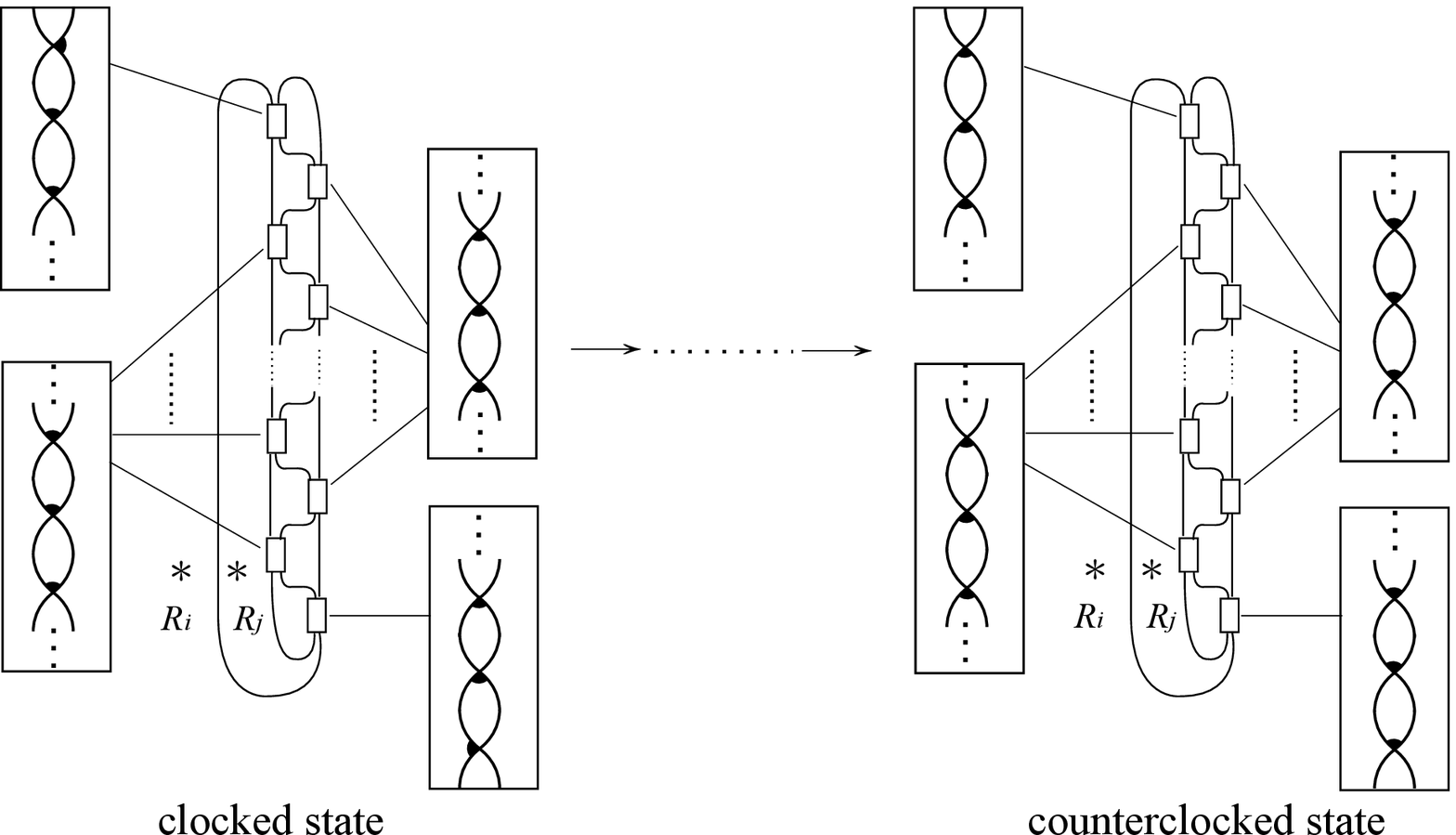}
 \caption{}
 \label{fig:main_theorem_figure_3}
 \end{center}
\end{figure}

Since this universe has two markers which are transposed once and $c(K)-2$ markers which are transposed twice, we have
\begin{equation*}
\begin{split}
 p(K) &\leq p(\tilde{K}; i, j) \\
      &= 1 + \frac{1 \times 2 + 2(c(K) - 2)}{2} \\
      &= c(K).
\end{split}
\end{equation*}
Using Theorem \ref{thm:relation_between_p_and_c}, we obtain
\[
p(K) = c(K).
\]
Similarly, we can prove that $p(K)$ is equal to $c(K)$ when $K$ has an odd number of boxes.
This complete the proof.

\end{proof}

The following example shows that if $K$ is not a prime knot, then the inequality of Theorem \ref{thm:relation_between_p_and_c} does not hold. 

\begin{example}
\rm
Let $K$ be a non-prime knot as in Figure \ref{fig:example_of_non_prime_knot_1}. 
\begin{figure}[h]
 \begin{center}
 \includegraphics[scale=0.4]{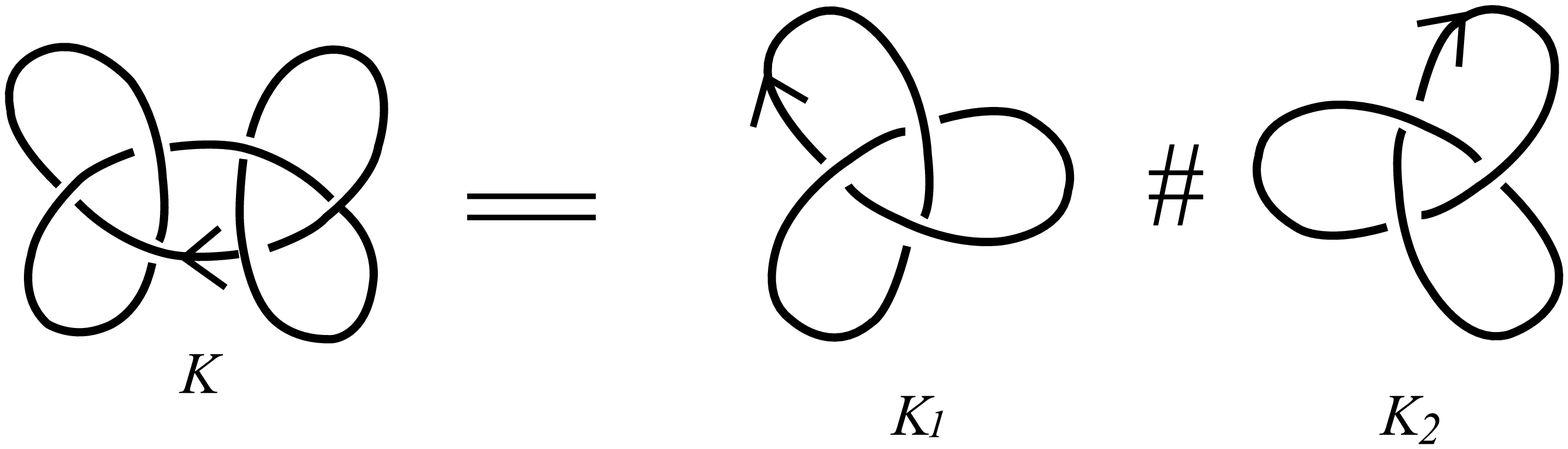}
 \caption{}
 \label{fig:example_of_non_prime_knot_1}
 \end{center}
\end{figure}
When we choose the universe and the starred regions as in Figures \ref{fig:example_of_non_prime_knot_2} and \ref{fig:example_of_non_prime_knot_3}, we can draw the lattice as in Figure \ref{fig:example_of_non_prime_knot_5}. 
\begin{figure}[h]
 \begin{center}
 \includegraphics[scale=0.4]{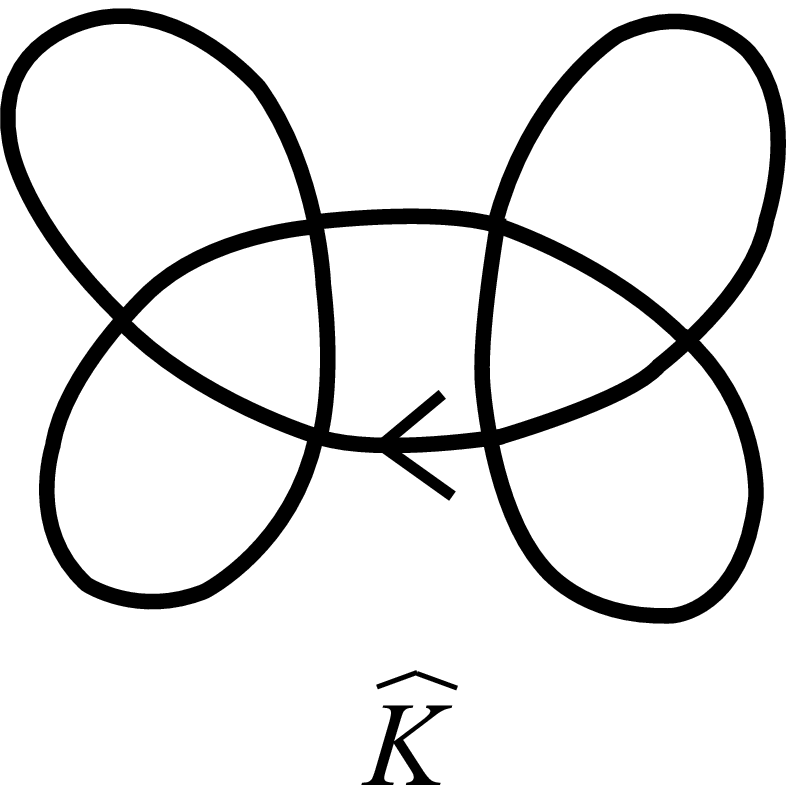}
 \caption{}
 \label{fig:example_of_non_prime_knot_2}
 \end{center}
\end{figure}
\begin{figure}[h]
 \begin{center}
 \includegraphics[scale=0.4]{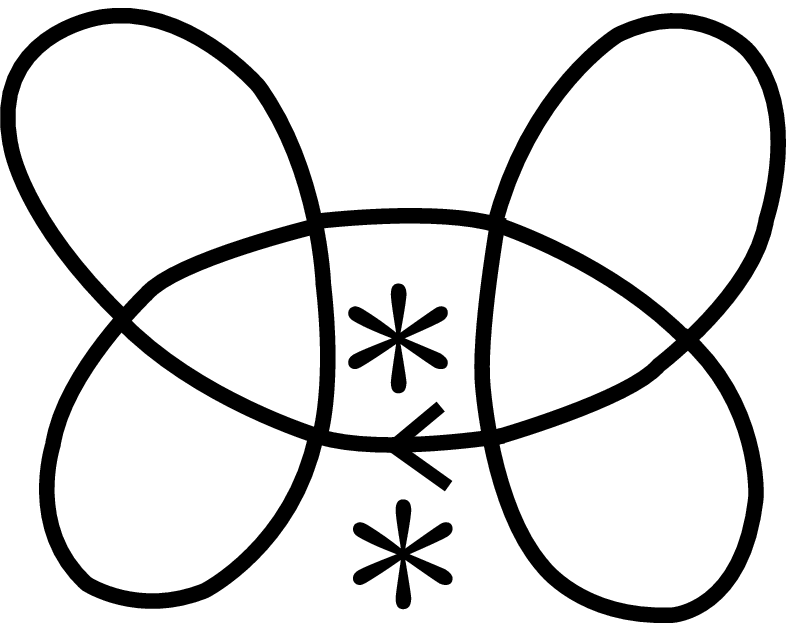}
 \caption{}
 \label{fig:example_of_non_prime_knot_3}
 \end{center}
\end{figure}
\begin{figure}[h]
 \begin{center}
 \includegraphics[scale=0.6]{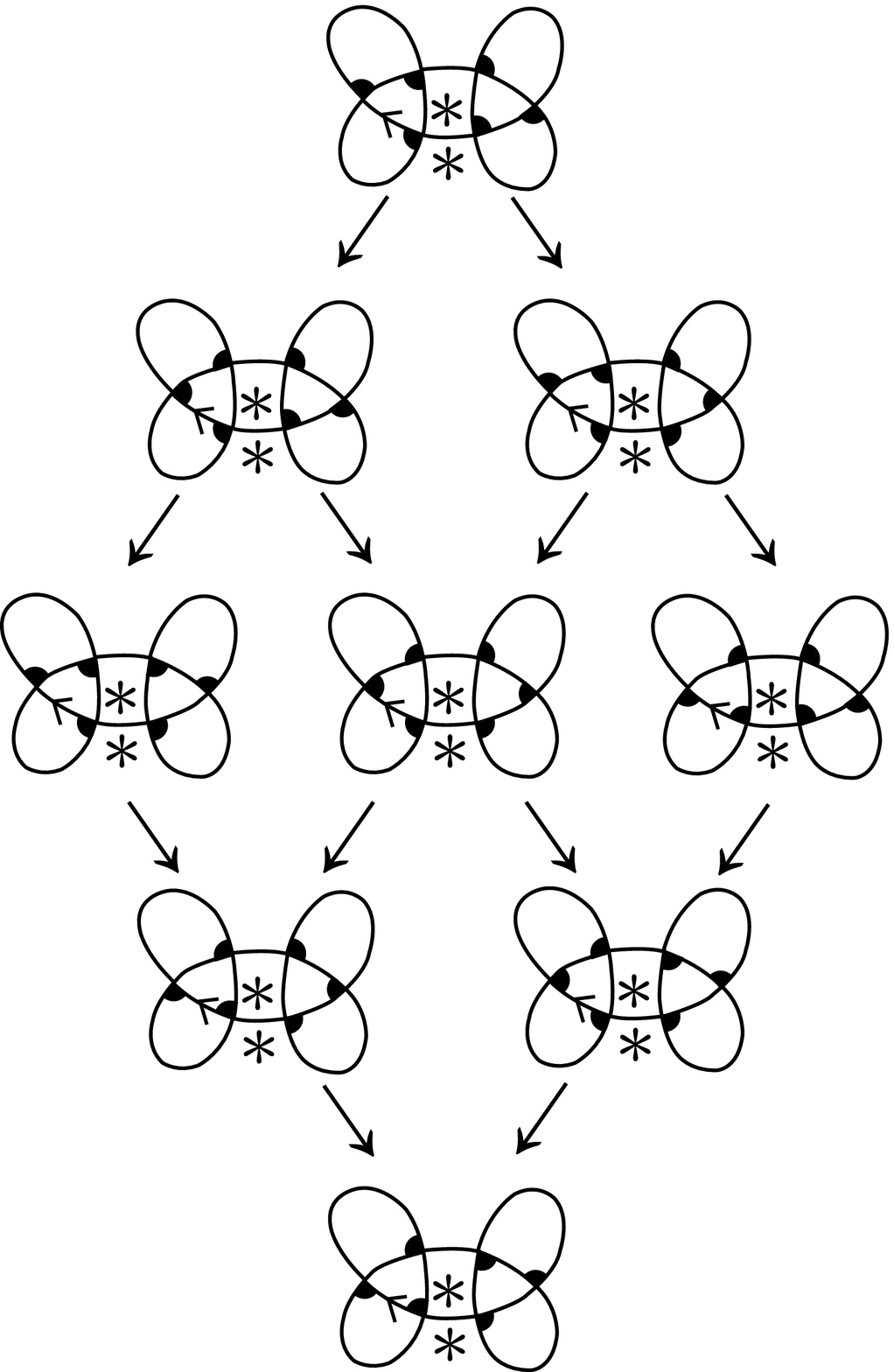}
 \caption{}
 \label{fig:example_of_non_prime_knot_5}
 \end{center}
\end{figure}

Therefore we see
\[
p(K) \leq 5 < 6 = c(K).
\]

\end{example}

\providecommand{\bysame}{\leavevmode\hbox to3em{\hrulefill}\thinspace}
\providecommand{\MR}{\relax\ifhmode\unskip\space\fi MR }
\providecommand{\MRhref}[2]{%
  \href{http://www.ams.org/mathscinet-getitem?mr=#1}{#2}
}
\providecommand{\href}[2]{#2}

\end{document}